\documentclass{amsart}
\usepackage{amsfonts}
\usepackage{amsmath,amscd}
\usepackage{amssymb}
\usepackage{amsthm}
\usepackage{newlfont}

\long\def\alert#1{\parindent2em\smallskip\hbox to\hsize
{\hskip\parindent\vrule
\vbox{\advance\hsize-2\parindent\hrule\smallskip\parindent.4\parindent
\narrower\noindent#1\smallskip\hrule}\vrule\hfill}\smallskip\parindent0pt}
 \newtheorem{thm}{Theorem}[section]
\newtheorem{cor}[thm]{Corollary}
 \newtheorem{lem}[thm]{Lemma}
 \newtheorem{prop}[thm]{Proposition}
\theoremstyle{definition}
 \newtheorem{defn}[thm]{Definition}
\theoremstyle{remark}

 \numberwithin{equation}{section}

\begin{document}

\title[Nilpotent Lie algebras of class $4$ ]
 {Nilpotent Lie algebras of class $4$ with the derived subalgebra of dimension $3$}

\author[F. Johari]{Farangis Johari}
\author[P. Niroomand]{Peyman Niroomand}
\author[M. Parvizi]{Mohsen Parvizi}

\address{Department of Pure Mathematics\\
Ferdowsi University of Mashhad, Mashhad, Iran}
\email{farangisjohary@yahoo.com, farangis.johari@mail.um.ac.ir}
\address{School of Mathematics and Computer Science\\
Damghan University, Damghan, Iran}
\email{niroomand@du.ac.ir, p$\_$niroomand@yahoo.com}
\address{Department of Pure Mathematics\\
Ferdowsi University of Mashhad, Mashhad, Iran}
\email{parvizi@math.um.ac.ir}



\date{\today}


\begin{abstract}
The paper is devoted to give a full classification of all  finite dimensional nilpotent Lie algebras $ L $ of class $4$ such that $ \dim L^2=3. $ Moreover, we classify the capable ones.

\end{abstract}

\maketitle

\section{Introduction}
It is well known that the classification of nilpotent Lie algebras is a classical problem. Several classifications of nilpotent Lie algebras of dimension at most $ 7 $  over various ground fields are available in the literature (See \cite{cic, Gr2, Gr}).
It is not easy to classify nilpotent Lie algebras with an arbitrary dimension. Hence we are interested to classify nilpotent  Lie algebras
by focusing on some other aspects rather than the dimension. For a given Lie algebra $ L $ with $ \dim L^2=1, $ the structure of $ L $ is given  in
\cite{ni54}. When $ \dim L^2=2, $  we gave the structure of $ L $  when $ L $ is of class $ 3 $ and with some restrictions  for class $ 2 $
in \cite{ni41}. The  purpose  of this paper is to describe a classification of all nilpotent Lie algebras of class $ 4 $ with the derived subalgebra of dimension $ 3. $ Moreover, in this class, we classify 
which ones are capable.
\section{ Preliminaries}
This section is devoted to give some elementary and known results that  will be needed for the next investigations.
All Lie algebras in this paper are finite dimensional over any arbitrary field.

First we  recall  the concept of a central product of two Lie algebras $A$ and $B.$

\begin{defn}\label{cent}
A Lie algebra $L$ is a central product of $A$ and $B,$ if $ L=A+B,$ where  $A$ and $B$ are ideals of $ L $ such that $ [A,B]=0$ and $A\cap B\subseteq Z(L).$ We denote the central product of two Lie algebras $A$ and $B$ by $A\dotplus B.$
\end{defn}

The  following lemma emphasizes the  Heisenberg Lie algebras are in fact  central products some of their ideals.

\begin{lem}\cite[Lemma 3.3]{pair}\label{fr}
Let $ L $ be a Heisenberg Lie algebra of dimension $2m+1.$ Then $ L $ is  a central product of its ideals  $B_j$ for all $ j, $ $1\leq j\leq m$ such that each $B_j$ is  the Heisenberg Lie algebra of dimension $3.$
\end{lem}

A Lie algebra $L$ is called capable provided that
$L \cong H/Z(H)$ for some Lie algebra $H.$  The
notion of the epicenter $Z^*(L)$ for a Lie algebra $L$ was defined in \cite{alam}. It is shown
that $L$ is capable if and only if $Z^*(L) = 0.$
Another notion having a relation to the capability is the concept of the exterior square of
Lie algebras, $L \wedge L,$ which was introduced in \cite{el}. Our approach is on the concept of
the exterior center $Z^{\wedge}(L),$ the set of all elements $l$ of $L$ for which $l\wedge l' = 0_{L\wedge L}$ for all
$l' \in L.$ Niroomand et al. in \cite{ni3} showed $Z^{\wedge}(L) = Z^*(L)$ for any finite dimensional
Lie algebra $L.$\\
It is not  an easy matter to determine  the capability
of a central product of Lie algebras in general, but the next result gives the answer to this question in a particular case.
\begin{prop}\label{Hi}\cite[Proposition 2.2]{ni4}
Let $L$ be a  Lie algebra such that $L=A\dotplus B$ with  $ A^2\cap B^2\neq 0.$ Then $ A^2\cap B^2\subseteq Z^{\wedge}(L)$.                 Moreover, $ L $ is non-capable.
\end{prop}
Let $ cl(L)$ denote the nilpotency class of a Lie algebra $ L. $ The following theorem gives the classification of all capable nilpotent Lie algebras of class $3$ with the
derived subalgebra of dimension $2.$

\begin{thm}\cite[Theorem 5.3]{ni41}\label{26117}
Let $ L $ be an $n$-dimensional Lie algebra    such that $cl(L)=3$  and $\dim L^2=2.$ Then $L $ is capable if and only if
 $L\cong L_{4,3}\oplus  A(n-4) $ or  $L\cong L_{5,5}\oplus  A(n-5).$
\end{thm}

From \cite{Gr}, the only Lie algebra of  maximal class of dimension $4$ is isomorphic to
\[L_{4,3}=\langle x_1,\ldots,x_4|[x_1, x_2] = x_3, [x_1, x_3] = x_4\rangle,\]
and there are  two Lie algebras of  maximal class of dimension $5$ that are isomorphic  to
\[L_{5,6}=\langle x_1,\ldots,x_5|[x_1, x_2] = x_3, [x_1, x_3] = x_4, [x_1, x_4] = x_5, [x_2, x_3] = x_5\rangle\]
and
\[L_{5,7}=\langle x_1,\ldots,x_5|[x_1, x_2] = x_3, [x_1, x_3] = x_4, [x_1, x_4] = x_5\rangle,\] respectively.

We say a Lie algebra $L$ is a semidirect sum of an ideal $I$ by a subalgebra $K$ if $L=I+K,$
$ I\cap K=0. $ The semidirect sum of an ideal $I$ by a subalgebra $K$ is denoted by $K\ltimes I.$\newline

 \begin{lem}\cite[Lemma 4.1]{ni41}\label{rr11}
  Let $L  $  be a $ 5$-dimensional  nilpotent stem Lie algebra of class $ 3 $ and $ \dim L^2=2. $ Then
\[ L\cong L_{5,5}=\langle x_1,\ldots,x_5| [x_1, x_2] = x_3, [x_1, x_3] = x_5, [x_2, x_4] = x_5\rangle.\]
Moreover,
$L_{5,5}=I\rtimes \langle x_4\rangle  $ in which  \[ I=\langle x_1,x_2,x_3,x_5| [x_1, x_2] = x_3, [x_1, x_3] = x_5\rangle\cong L_{4,3},~\text{and}~[I, \langle x_4\rangle]=\langle x_5\rangle=Z(L_{5,5}).\]
  \end{lem}
 \begin{lem}\label{rr112} Let $L  $  be a $ 6$-dimensional  nilpotent stem Lie algebra of class $ 4 $ and $ \dim L^2=3. $ Then $L  $  is isomorphic to one of the Lie algebras listed below.
  \begin{itemize}
  \item[$(1)$]
 $ L_{6,11}=\langle x_1,\ldots,x_6| [x_1, x_2] = x_3, [x_1, x_3] = x_4, [x_1, x_4] = [x_2, x_3] =[x_2, x_5]= x_6\rangle =I_1\rtimes \langle x_5\rangle,  $ in which  $ I_1=\langle x_1,x_2,x_3,x_4,x_6| [x_1, x_2] = x_3, [x_1, x_3] = x_4,[x_1, x_4] = [x_2, x_3] = x_6\rangle\cong L_{5,6}$ and $[I, \langle x_5\rangle]=\langle x_6\rangle=Z(I_1).$
 \item[$(2)$]
 $ L_{6,12}=\langle x_1,\ldots,x_6| [x_1, x_2] = x_3, [x_1, x_3] = x_4, [x_1, x_4] =[x_2, x_5]=x_6\rangle =I_2\rtimes \langle x_5\rangle,  $ in which  $ I_2=\langle x_1,x_2,x_3,x_4,x_6| [x_1, x_2] = x_3, [x_1, x_3] = x_4,[x_1, x_4] =  x_6\rangle\cong L_{5,7}$ and $ [I_2,\langle x_5\rangle]=\langle x_6\rangle =Z(I_2). $
 \item[$(3)$]
 $ L_{6,13}=\langle x_1,\ldots,x_6| [x_1, x_2] = x_3, [x_1, x_3] = [x_2, x_4]=x_5, [x_1,x_5]=[x_3, x_4]=x_6\rangle =I_3\rtimes \langle x_4\rangle,  $ in which  $ I_3=\langle x_1,x_2,x_3,x_5,x_6| [x_1, x_2] = x_3, [x_1, x_3] = x_5,[x_1, x_5] =  x_6\rangle\cong L_{5,7}$
 and $ [I_3,\langle x_4\rangle]=\langle x_5,x_6\rangle. $
 \end{itemize}
  \end{lem}
  \begin{proof}
     By looking at the classification  of   nilpotent Lie algebras of dimension  at most $ 6 $ in \cite{cic,Gr}, we get $L\cong L_{6,11},$ $L\cong L_{6,12}$ or  $L\cong L_{6,13}.$  Let $L\cong L_{6,11}.$ It is easy to check that $ L_{6,11}=I_1\rtimes \langle x_5\rangle,  $ in which  $ I_1=\langle x_1,x_2,x_3,x_4,x_6| [x_1, x_2] = x_3, [x_1, x_3] = x_4,[x_1, x_4] = [x_2, x_3] = x_6\rangle\cong L_{5,6}$ and $[I, \langle x_5\rangle]=\langle x_6\rangle=Z(L_{5,6}).$ Similarly, we can see that
   $ L_{6,12}=\langle x_1,\ldots,x_6| [x_1, x_2] = x_3, [x_1, x_3] = x_4, [x_1, x_4] =[x_2, x_5]=x_6\rangle =I_2\rtimes \langle x_5\rangle,  $ in which  $ I_2=\langle x_1,x_2,x_3,x_4,x_6| [x_1, x_2] = x_3, [x_1, x_3] = x_4,[x_1, x_4] =  x_6\rangle\cong L_{5,7}$ and $ [I_2,\langle x_5\rangle]=\langle x_6\rangle=Z(L_{5,7}). $ Now, let $L\cong L_{6,13}.$
   Clearly $Z(L) = \langle x_6\rangle$ and $L_{6,13} = I_3+\langle x_4\rangle,$ where $ I_3=\langle x_1,x_2,x_3,x_5,x_6| [x_1, x_2] = x_3, [x_1, x_3] = x_5,[x_1, x_5] =  x_6\rangle\cong L_{5,7}$ and $ [I_3,\langle x_4\rangle]=\langle x_5,x_6\rangle,$ as required.
  \end{proof}

We need the following lemma.
\begin{lem}\label{z}\cite[Lemma 1]{zac} Let $L$ be a nilpotent Lie algebra and $H$ be a subalgebra of $L$ such that $L^2 =
H^2 + L^3.$ Then
$L^i = H^i$ for all $i \geq 2.$
Moreover, $H$ is an ideal of $L.$
\end{lem}

\section{Main results}
We are going  to give the structure of all nilpotent Lie algebras of class $ 4$  with the derived subalgebra of dimension $ 3. $
Moreover, we determine which one of these are capable.

 The next two results which are stated for  Lie algebras have a group theoretical reason for $p$-groups  in  \cite[Lemma 2.3 and Theorem 2.4]{bl}.  Here, we give a proof for them.

\begin{lem}\label{12}
Let $ L $ be an $n$-dimensional nilpotent Lie algebra of class $ c $ such that $ \dim L^2=c-1 $ and $ I $ be an ideal of dimension $i$ $( 0\leq i\leq c-1)$ contained in $ L^2 .$ Then $ I=L^{c-i+1}.$
\end{lem}
\begin{proof}
Clearly, $ \dim L^j =  c-j+1,$ where $ 1\leq j \leq c $. We proceed by induction on $ c-i+1 $. If $ i=c-1, $ the result follows easily. Let $ c-i>1$ and $ M/I $  be an ideal of dimension $ 1 $ such that $ M/I \subseteq  (L/I)^2 \cap Z(L/I).$ So $ M $ is an $( i+1)$-dimensional ideal of $ L $     such that $ I\subsetneqq M\subseteq L^2. $ By using the induction hypothesis, $ M=L^{c-i }.$ Since $M/I\subseteq Z( L/I),$ we have  $L^{c-i+1}= [L,L^{c-i}]=[L,M]\subseteq I.$  Now, both   $ I $ and $ L^{c-i+1} $ are  of dimension $ i, $ and hence  $ I=L^{c-i+1} $. The result follows.
\end{proof}

 Recall that an $n$-dimensional nilpotent Lie algebra  $L$  is said to be  nilpotent of maximal class if $cl(L)=n-1.$  For a maximal class Lie algebra $L,$ we have  $\dim (L/L^2) = 2,$ $Z_i(L) = L^{n-i}$ and $\dim (L^j/L^{j+1}) = 1$ for all $i,$  $ j,$  $0\leq  i \leq n-1$ and   $2\leq j \leq n-1$ (see \cite{bos} for more information).
\begin{prop}\label{13}
Let $ L $ be an $n$-dimensional Lie algebra of class $ c $ such that $ \dim L^2=c-1.$ Then $ L^2 \cap Z_i(L)=L^{c-i+1}$ for all  $ i, $ $ 0\leq i\leq c-1.$
\end{prop}
\begin{proof}
By contrary, let $ L^{c-i+1} \subsetneqq L^2 \cap Z_i(L).$ Since $\dim L^{c-i+1}=i,  $ we get $k =\dim ( L^2 \cap Z_{i}(L))\geq i+1. $ Lemma \ref{12} implies $ L^2 \cap Z_{i}(L)=L^{c-k+1} \subseteq L^{c-i+1}.$ Since $\dim L^{c-i+1}=i,$ we have a contradiction. This completes the proof.
\end{proof}
Recall that from \cite{mon} a Lie algebra $ S $ is called a stem Lie algebra if  $ Z(S)\subseteq S^2.$
Now, we are able to prove the following result which is useful in the rest.
\begin{prop}\label{134}
Let $ L $ be an $n$-dimensional Lie algebra of class $ c $ such that $ \dim L^2=c-1. $ Then $L$ is   stem if and only if $  Z(L)=L^c\cong A(1).$
\end{prop}
\begin{proof}
Let $L$ be  stem. Then by Proposition \ref{13}, we have \[L^{c}\subseteq  Z(L)\subseteq L^2\cap Z(L)=L^{c}.\] This implies $ \dim Z(L)=\dim L^{c}= 1,$ as required.
\end{proof}

\begin{thm}\label{48}
Let $L$ be an $n$-dimensional nilpotent stem Lie algebra  of class $4$  $(n\geq 6)$ and $ \dim L^2=3 $ such that $L=I+K,$ where  $ K $ is a subalgebra of $L$ and  $I$ is a maximal class  subalgebra of dimension $ 5 $ such that $[I,K]\subseteq Z(I)=Z(L).$ Then
\begin{itemize}
\item[$(i)$] If $K$ is a non-trivial abelian Lie algebra such that $K\cap I=0,$ then
$[I,K]=Z(L)  $ and $K\cong A(1).$ Moreover, $L=I\rtimes K$
and $ L $ is isomorphic to  $ L_{6,11}$ or $ L_{6,12}.$
\item[$(ii)$] If $ K $ is a Heisenberg Lie algebra  and $I\cap K=Z(L)\cong A(1),$ then $ L $ is isomorphic to one of the following Lie algebras.
\begin{itemize}
\item[$(a)  $]$L\cong I\dotplus K\cong \langle x_1,\ldots,x_5,a,b|[x_1, x_i] = x_{i+1}, [x_2, x_3] =[a,b]= x_5,1\leq i\leq 4\rangle,$ where $I\cong  L_{5,6}$ and $ K\cong H(1). $
\item[$(b)  $]$L\cong I\dotplus K\cong \langle x_1,\ldots,x_5,a,b|[x_1, x_i] = x_{i+1}, [a,b]=x_5,1\leq i\leq 4\rangle,$ where $I\cong  L_{5,7}$ and $ K\cong H(1). $
\item[$(c)  $]$L\cong I\dotplus K\cong I_1\dotplus K_1,$ where $ I_1\cong \langle x_1,\ldots,x_5,a,b|[x_1, x_i] = x_{i+1}, [x_2, x_3] $  $                                                                                                                                                                                        =[a,b]=x_5,1\leq i\leq 4\rangle,$  $ K_1\cong H(m-1) $ and $ n=2m+5$ for all $ m\geq 2. $
\item[$(d)  $]$L\cong I\dotplus K\cong I_2\dotplus K_2,$ where $I_2\cong \langle x_1,\ldots,x_5,a,b|[x_1, x_i] = x_{i+1}, [a,b]=x_5,1\leq i\leq 4\rangle,$  $ K_2\cong H(m-1) $ and $ n=2m+5$ for all $ m\geq 2. $
\end{itemize}
\item[$(iii)  $]If $ A(1) \cong I\cap K=Z(L)=K^2\neq Z(K),$ then $L\cong (I\rtimes A)\dotplus K,$ where  $ K\cong H(m) $ and $A\cong A(1),$  $ [I,A]=Z(L)=Z(I)=K^2$ and $ n=2m+6.$ Moreover, $ I\rtimes A\cong L_{6,11} $ or $ I\rtimes A\cong L_{6,12}. $
\end{itemize}                                           
\end{thm}
\begin{proof}                                                                                                                                                                                                                                                                                                                                                                                                                                                                                                                                                                                                                                                                                                                                                                                                                                                                    
\begin{itemize}
\item[ $(i)$]  We have  $[I,K]\subseteq Z(I)=Z(L),$ so $ I$ is an ideal of $L$ and  $I\cong  L_{5,6}$ or $I\cong  L_{5,7}.$  We know that $ I $ is a Lie algebra of maximal class of dimension $ 4 $ so $ Z(I)=Z(L)=I^4. $  We also have $ \dim L^2=\dim I^2=3. $ Therefore  $ L^2=I^2. $ Since $ K\cap I=0, $ we obtain  $ \dim K=\dim L-\dim I=n-5 $ and so $L=I\rtimes K,$  in which $K\cong A(n-5).$ We claim that $ [I,K ]=Z(I).$  By contrary, assume that  $ [I,K ]=0.  $ Since $ K $ is abelian, we have $ K\subseteq Z(I)=Z(L)\subseteq I.$ Now  $ I\cap K\neq 0,$ so we have  a contradiction. Thus $ [I,K ]=Z(I)=Z(L).$ \\ We claim that $K\cong A(1).  $ \newline
 First assume that $ n=6.$  Lemma \ref{rr112} implies $ L\cong I_1\rtimes K\cong L_{6,11}$ or $L\cong I_2\rtimes K\cong L_{6,12},  $ in which
 $ I_1\cong L_{5,6},$    $ I_2\cong L_{5,7},$  $ [I_1,K]=Z(L)$ and $[I_2,K]=Z(L).$\\
 Now, let $ n\geq 7$  and $ K=\bigoplus_{i=1}^{n-5}\langle z_i\rangle. $ In this case we show that $ K\cap I\neq 0, $ which is
 impossible, so this case cannot occur. By using the Jacobi identity, for all  $ i, $ $ 1\leq i\leq n-5, $ we have
 \begin{align*}
 [z_i,x_3]=[z_i,[x_1, x_2]]=[z_i,x_1, x_2]+[x_2,z_i, x_1]=0
 \end{align*}
 and
 \begin{align*}
 [z_i,x_4]=[z_i,[x_1, x_3]]=[z_i,x_1, x_3]+[x_3,z_i, x_1]=0
 \end{align*}
 since $ [z_i,x_1], [z_i,x_3]$  and $[x_2,z_i] $ are central. Thus $[z_i,x_3]=[z_i,x_4]= 0$ for all $ i, $ $ 1\leq i\leq n-5.$ Now, let $[z_i,x_1]= \alpha_i x_6$ and $ \alpha_i\neq  0 $ for all $ i, $ $ 1\leq i\leq n-5.$ Putting $ z_i'=z_i+\alpha_i x_4,$ we have $[z_i',x_1]=[z_i+\alpha_i x_4,x_1]=[z_i,x_1]+\alpha_i [x_4,x_1]= \alpha_i x_6-\alpha_i x_6=0. $ Also we obtain $ [z_i',x_4]=0.$ Thus $[z_i',x_1]=[z_i',x_3]=[z_i',x_4]= 0$ for all $i$, $ 1\leq i\leq n-5.$
  Now let $ [z_i',x_2]= \alpha x_6$ and  $ [z_j',x_2]=\beta x_6,$ in which $ \alpha\neq 0 $ and $ \beta\neq 0 $  for some fixed $ i, j, $ with $ i\neq j.$
 Put $ d_{ij}=\beta z_i'-\alpha z_j'.$ We have
 \begin{align*}
 [d_{ij},x_2]=[\beta z_i'-\alpha z_j',x_2]=\beta \alpha x_6-\beta \alpha x_6=0
 \end{align*}
and so  $[d_{ij},x_2]=0.$ On the other hand, $ [d_{ij},x_1]=[d_{ij},x_2]=[d_{ij},x_3]=[d_{ij},x_4]=0. $ Therefore $ [d_{ij},I]=0 $ and hence $ d_{ij}\in Z(L)=Z(I)= \langle x_6\rangle.$ Since
 \begin{align*}
 &d_{ij}=\beta z_i'-\alpha z_j'=\beta(z_i-\alpha_i x_4)-\alpha ( z_j-\alpha_j x_4)\\&=\beta z_i- \alpha z_j + (\alpha\alpha_j-\beta\alpha_i)x_4\in Z(I),
 \end{align*}
  $0\neq  \beta z_i- \alpha z_j\in K\cap I=0,$ which is a contradiction.  Thus $ n=5, $ $ K\cong A(1), $   $L=I\rtimes \langle z_1\rangle$ and $ [x_2,z_1]=x_6, $ as required. Now, considering the classification of nilpotent Lie algebras of dimension $ 6$ with $ \dim L^2=3 $ which is given in \cite{Gr}, using Lemma \ref{rr112} and our assumption, we should have $L\cong L_{6,11}$ or $L\cong L_{6,12}.$
  \item[$ (ii) $]Since  $I\cap K=Z(L)\cong A(1),$  we have $I\cap K=Z(L)=Z(K) $ and $\dim (K)=\dim (L)-\dim (I)+\dim (I\cap K)=n-5+1=n-4. $    Now since   $2m+1=\dim (K)=n-4,$ we have $ n=2m+5. $ We are going to show that  $ [I,K]=0. $ In fact, we show that there exists $I_1\cong  L_{5,6}$ or $I_1\cong  L_{5,7}$ and $ K_2\cong H(m)$ with $ [I_1,K_2]=0 $ and $ L=I_1\dotplus K_2. $
 First let $ m=1. $ We have $ \dim L=6 $ and $ K=\langle x,y,x_6|[x,y]=x_6\rangle, $ since $ K\cong H(1).$
  By using the Jacobi identity, we have
 \begin{align*}
& [x,x_3]=[x,[x_1, x_2]]=[x,x_1, x_2]+[x_2,x, x_1]=0,\\&
 [x,x_4]=[x,[x_1, x_3]]=[x,x_1, x_3]+[x_3,x, x_1]=0,\\&
  [y,x_3]=[y,[x_1, x_2]]=[y,x_1, x_2]+[x_2,y, x_1]=0
  \end{align*}
 and
 \begin{align*}
 [y,x_4]=[y,[x_1, x_3]]=[y,x_1, x_3]+[x_3,y, x_1]=0.
 \end{align*}
 Thus $ [x,x_3]=[x,x_4]=[y,x_3]=[y,x_4]=0. $
  Now, let $ [x,x_1]= \alpha x_6$ and  $ [y,x_1]=\beta x_6$ for some scalars $ \alpha  $ and $ \beta.  $ Then by taking $ x'=x+\alpha x_4 $ and $ y'=y+\beta x_4,$
 we have $[x',x_1]=[y',x_1]=0.  $  Without loss of
generality, let $[y',x_2]=0 $ and assume that $[x',x_2]=\eta x_6$ for a scalar $ \eta. $ By taking
$x_2'=x_2-\eta y',  $ we have $ [x',x_2']=0. $ Thus
 $ L=I\dotplus K, $ where $ K\cong H(1) $ and $ I $ is a maximal class Lie algebra of dimension $ 5. $
Thus $ L\cong \langle x_1,\ldots,x_5,a,b|[x_1, x_i] = x_{i+1}, [x_2, x_3] =[a,b]=x_5,1\leq i\leq 4\rangle$
or
$L\cong \langle x_1,\ldots,x_5,a,b|[x_1, x_i] = x_{i+1}, [a,b]=x_5,1\leq i\leq 4\rangle.$ It completes
 cases $ (a) $ and $ (b)$ of $ (ii).$\\
  Now, let $ m\geq 2$ and $ H(m)=\langle a_1,b_1,\ldots, a_m,b_m,z\big{|}[a_l,b_l]=z,1\leq
l\leq m\rangle.$ Lemma \ref{fr} implies that $H(m)=T_1\dotplus \ldots \dotplus T_m,$ in which  $T_i\cong H(1) $
 for all  $i, 1\leq i \leq m.$ With the same procedure used in  case $ (i) $ and changing the variables we can see that $ [T_i,I]= 0$ for all $i$, $ 1\leq i\leq m.$
So $[I,K]=0$ and hence  $ L=I \dotplus K.$ Since $ m\geq 2, $ we have $ L=(I\dotplus T_1)\dotplus (T_2 \dotplus \ldots \dotplus T_m)$ such that
$ I\dotplus T_1\cong \langle x_1,\ldots,x_5,a,b|[x_1, x_i] = x_{i+1}, [x_2, x_3] =[a,b] =x_5,1\leq i\leq 4\rangle,$  $ K_1\cong H(m-1) $
or
$I\dotplus T_1\cong \langle x_1,\ldots,x_5,a,b|[x_1, x_i] = x_{i+1}, [a,b]=x_5,1\leq i\leq 4
\rangle$ and $T_2 \dotplus \ldots \dotplus T_m\cong H(m-1).  $ It completes cases $ (c) $ and $ (d)$ of $ (ii).$
\item[$ (iii) $] Since
 $I\cap K=K^2=Z(L)=Z(I)\cong A(1),$  $\dim (K)=\dim (L)-\dim (I)+\dim (I\cap K)=n-5+1=n-4. $ We know that $ \dim  K^2=1, $ so \cite[Theorem 3.6]{ni3} implies  $ K\cong K_1\oplus A, $ in which $K_1\cong H(m)$ and $A\cong A(n-2m-5).$
  Since $ A\neq 0,$ we have  $n\neq 2m-5.$ Thus $L=I+( K_1\oplus A)$ such that $[I,K]\subseteq Z(L)=Z(I).$ We are going to show that $ A\cong A(1), $ $ [I,K_1]=0 $ and $ [I,A]=Z(I)=Z(L). $ Similar to  part $ (ii), $ we can see that  $[I,K_1]=0.$ We claim that  $ [I,A]\neq 0.  $ By contrary, let  $[A,K_1]=[I,A]=0,  $ so $ A\subseteq Z(L)=Z(I). $ Since  $ A\cap I=0,$ we have $ A=0, $ which is a contradiction. So we have $ [I,A]=Z(L).$ We claim that $ \dim A=1. $ Let  $ \dim A\geq 2. $ Similar to the proof of  part $ (i),$ we have    $ [a_1,x_1]=[a_2,x_1]=[a_1,x_3]=[a_2,x_3]=[a_1,x_4]=[a_2,x_4]=0$  where $ a_1,a_2\in A $ and $a_1\neq a_2.  $ Now let $[a_1,x_2]= \alpha x_5$ and  $ [a_2,x_2]=\beta x_5$ for some non-zero scalars  $ \alpha $ and $ \beta. $
 Putting $ a_1'=\beta a_1-\alpha a_2.$ We have
$ [a_1',x_2]=[\beta a_1-\alpha a_2,x_2]=\beta \alpha x_5-\beta \alpha x_5=0$
and so  $[a_1',x_2]=0.$ Hence $ [a_1',x_1]=[a_1',x_2]=[a_1',x_3]=[a_1',x_4]=0. $ Therefore $ [a_1',I]=0 $ and hence $ a_1'\in Z(L)=Z(I)= \langle x_5\rangle=K_1^2.$ So $ a_1'\in K_1 $ and since $ K_1\cap A=0,$ we have  a contradiction. Hence $ A\cong A(1) $ and so $n=2m+6.$
   Thus $ L=(I\rtimes A)\dotplus K_1 $ such that  $[I,A]=Z(L)=Z(I).$ By using part $(i),$ we have $ I\rtimes A\cong L_{6,11} $ or $ I\rtimes A\cong L_{6,12}. $ The proof of case $ (iii) $ is completed.
\end{itemize}
\end{proof}
\begin{prop}\label{481}
Let $L$ be an $n$-dimensional nilpotent stem Lie algebra  of class $4$  $(n\geq 7)$ and $ \dim L^2=3 $ such that $L=I+K,$ where $I$ and $ K $ are two subalgebras of $L,$  $I\cong L_{6,13}$  and  $[I,K]\subseteq Z(I)=Z(L)\cong A(1).$ Then
\begin{itemize}
\item[$(i)$] If $K$ is a non-trivial abelian Lie algebra such that $K\cap I=0,$ then
$[I,K]=Z(L)  $ and $K\cong A(1).$ Moreover, $L=I\rtimes K\cong  \langle x_1,\ldots,x_6,x_7| [x_1, x_2] = x_3, [x_1, x_3] = [x_2, x_4]=x_5, [x_1,x_5]=[x_3, x_4]=[x_2,x_7]=[x_4,x_7]=x_6\rangle.$
\item[$(ii)$] If $ K $ is a Heisenberg Lie algebra  and $I\cap K=Z(L)\cong A(1),$ then $ L $ is isomorphic to one of the following Lie algebras.
\begin{itemize}
\item[$(a)  $]  $L\cong I\dotplus K\cong \langle x_1,\ldots,x_6,a,b| [x_1, x_2] = x_3, [x_1, x_3] = [x_2, x_4]=x_5, [x_1,x_5]=[x_3, x_4]=[a,b]=x_6\rangle,$
 where $ K\cong H(1). $
  \item[$(b)  $]  $L\cong I\dotplus K\cong I_1\dotplus I_2,$ where $ I_1\cong\langle x_1,\ldots,x_6,a,b$  $| [x_1, x_2] = x_3, [x_1, x_3] = [x_2, x_4]=x_5, [x_1,x_5]=[x_3, x_4]=[a,b]=x_6\rangle,$  $ I_2\cong H(m-1)$ and $n=2m+6$ for all $ m\geq 2. $
\end{itemize}
\item[$(iii)  $]If $ A(1) \cong I\cap K=Z(L)=K^2\neq Z(K),$  then $L=(I\rtimes A)\dotplus K,$ where  $ K\cong H(m), $  $A\cong A(1),$  $ [I,A]=Z(L)=Z(I)=K^2$ and $ n=2m+7.$ Moreover, $ I\rtimes A\cong \langle x_1,\ldots,x_6,x_7| [x_1, x_2] = x_3, [x_1, x_3] = [x_2, x_4]=x_5, [x_1,x_5]=[x_3, x_4]=[x_2,x_7]=[x_4,x_7]=x_6\rangle. $
\end{itemize}
\end{prop}
\begin{proof}
It is similar to the proof of Theorem \ref{48}.
\end{proof}

W are ready to determine the central factors of all stem Lie algebras $ T $ such that $cl(T)=4$ and $\dim T^2=3.$

  \begin{lem}\label{ggg}
   Let $ T $ be an $ n $-dimensional stem Lie algebra   such that $cl(T)=4$ and $\dim T^2=3.$ Then $T/Z(T)\cong L_{4,3}\oplus  A(n-4) $ or  $T/Z(T)\cong L_{5,5}\oplus  A(n-5).$
  \end{lem}
  \begin{proof}
By using Proposition \ref{134},  we have  $T^2/Z(T)\cong  A(2).$
Since $ T/Z(T) $ is capable, Theorem \ref{26117} implies that $T/Z(T)\cong L_{4,3}\oplus  A(n-4) $ or  $T/Z(T)\cong L_{5,5}\oplus  A(n-5).$ It completes the proof.
  \end{proof}
 In the following theorem, we determine the structure of all stem Lie algebras of class $ 4 $ with the derived subalgebra of dimension $ 3.$
  \begin{thm}\label{lkl}
    Let $ T $ be an $ n $-dimensional stem Lie algebra   such that $cl(T)=4$ and $\dim T^2=3.$ Then $ T $ is isomorphic to one of the following Lie algebras.
\begin{itemize}
\item[$(a)$] $T\cong L_{5,6}.$
 \item[$(b)$]$T\cong L_{5,7}.$
\item[$ (c) $]
$ T\cong  I\rtimes K \cong L_{6,11},$  where $I\cong L_{5,6},$ $K\cong A(1),$  $ [I,K]=Z(I)=Z(T) $ and $Z_2( T) =Z_2( I)\rtimes K. $
\item[$ (d) $]
$ T\cong  I\rtimes K \cong L_{6,12},$  where $I\cong L_{5,7},$ $K\cong A(1),$   $ [I,K]=Z(I)=Z(T) $ and $Z_2( T) =Z_2( I)\rtimes K. $
\item[$(e)  $]$T\cong I_1\dotplus I_2\cong \langle x_1,\ldots,x_5,a,b|[x_1, x_i] = x_{i+1}, [x_2, x_3] =[a,b]= x_5,1\leq i\leq 4\rangle,$ where $I_1\cong  L_{5,6}$ and $ I_2\cong H(1). $
\item[$(f)  $]$T\cong I_1\dotplus I_2\cong \langle x_1,\ldots,x_5,a,b|[x_1, x_i] = x_{i+1}, [a,b]=x_5,1\leq i\leq 4\rangle,$ where $I_1\cong  L_{5,7}$ and $ I_2\cong H(1). $
\item[$(g)  $]$T\cong  I_1\dotplus K_1,$ where $ I_1\cong \langle x_1,\ldots,x_5,a,b|[x_1, x_i] = x_{i+1}, [x_2, x_3] =[a,b] =x_5,1\leq i\leq 4\rangle,$  $ K_1\cong H(m-1) $ and $ n=2m+5$ for all $ m\geq 2. $
\item[$(h)  $]$T\cong I_2\dotplus K_2,$ where $I_2\cong \langle x_1,\ldots,x_5,a,b|[x_1, x_i] = x_{i+1}, [a,b]=x_5,1\leq i\leq 4\rangle,$  $ K_2\cong H(m-1) $ and $ n=2m+5$ for all $ m\geq 2. $
 \item[$ (k )$]
$ T\cong ( I\rtimes  K)\dotplus I_3 ,$ where $ I_3\cong H(m),$  $K\cong A(1),$ $ [I,K]=Z(I)=Z(T), $ $I\cong L_{5,6},$  $ I\rtimes  K\cong L_{6,11},$
$Z_2( T) =(Z_2( I)\rtimes  K)\dotplus I_3$  and $n=2m+6.$
\item[$ (l)$]
$ T\cong ( I\rtimes  K)\dotplus I_4,$ where $ I_4\cong H(m),$  $K\cong A(1),$ $ [I,K]=Z(I)=Z(T), $ $I\cong L_{5,7},$ $I\rtimes  K\cong L_{6,12},$
$Z_2( T) =(Z_2( I)\rtimes  K)\dotplus I_4$  and $n=2m+6.$
\item[$(m)  $] $T\cong L_{6,13}.$
\item[$ (n )$]$ T\cong \langle x_1,\ldots,x_6,x_7| [x_1, x_2] = x_3, [x_1, x_3] = [x_2, x_4]=x_5, [x_1,x_5]=[x_3, x_4]=[x_2,x_7]=[x_4,x_7]=x_6\rangle.$
\item[$(p)$]    $T=I_5\dotplus K\cong \langle x_1,\ldots,x_6,a,b| [x_1, x_2] = x_3, [x_1, x_3] = [x_2, x_4]=x_5, [x_1,x_5]=[x_3, x_4]=[a,b]=x_6\rangle, $  where $I_5\cong L_{6,13}$ and
 $ K\cong H(1). $
 \item[$ (q) $]$T\cong I_6\dotplus I_7,$ where $ I_6\cong\langle x_1,\ldots,x_6,a,b| [x_1, x_2] = x_3, [x_1, x_3] = [x_2, x_4]=x_5, [x_1,x_5]=[x_3, x_4]=[a,b]=x_6\rangle$ and $ I_7\cong H(m-1)$  for all $ m\geq 2. $
\item[$(r)  $] $T\cong (I_8\rtimes A)\dotplus K_3,$ where  $ K_3\cong H(m), $ $ I_8\cong L_{6,13},$ $A\cong A(1),$  $ [I_8,A]=Z(I_8)=K^2$ and $ n=2m+7.$ Moreover, $ I_8\rtimes A\cong \langle x_1,\ldots,x_7| [x_1, x_2] = x_3, [x_1, x_3] = [x_2, x_4]=x_5, [x_1,x_5]=[x_3, x_4]=[x_2,x_7]=[x_4,x_7]=x_6\rangle. $
\end{itemize}
  \end{thm}
\begin{proof}
Since $ cl(T)=4, $ we have $\dim T\geq 5.  $ If $\dim T=5,$ then $ T $ is a   $5$-dimensional Lie algebra of maximal class  and hence $T\cong L_{5,6}$ or $T\cong L_{5,7}.$  \\ Assume that $\dim T\geq 6.$
 Thus $ \dim T/Z(T)\geq 5 $ and so $ \dim I_2\geq 2. $
Proposition \ref{134} implies $T/Z(T)\cong L_{4,3}\oplus  A(n-4) $ or  $T/Z(T)\cong L_{5,5}\oplus  A(n-5)$  and $Z(T)=T^4\cong A(1).$
 First let $T/Z(T)\cong L_{4,3}\oplus  A(n-4). $
There exist two ideals $ I_1/Z(T)$ and $ I_2/Z(T)$   of $ T/Z(T)$ such that
$ I_1/Z(T)\cong L_{4,3}$ and
$_2/Z(T)\cong A(n-5).$
Since $ T^2/Z(T)=\big{(}I_1^2+Z(T)\big{)}/Z(T),  $ we have $ T^2=I_1^2+Z(T)  $ and $ Z(T)=T^4.$
Using Lemma \ref{z}, we have $ T^2=I_1^2 $ and so $cl(T)=cl(I_1)=4.$ Hence $ I_1$ is a $5$-dimensional Lie algebra of maximal class  and so    $I_1\cong L_{5,6}$ or $I_1\cong L_{5,7}.$ Now, $ Z(T)=Z(I_1)$ because $Z(T)\cap   I_1\subseteq Z( I_1)  $ and $ \dim Z(T)=1. $ Since $ Z(T)\subseteq   I_1\cap I_2 \subseteq Z(T),$ we have $ I_1\cap I_2= Z(T)=Z(I_1). $
Hence
 \[Z_2(T)/Z(T)=Z(T/Z(T))=Z(I_1/Z(T))\oplus I_2/Z(T)=Z_2(I_1)/Z(T)+I_2/Z(T).\]  Since $ I_1 $ is maximal class of dimension $ 5, $ we have
   $Z_2( T) =Z_2( I_1)+I_2=I_1^3+I_2.$
Now, we are going to determine the structure of $ I_2. $ Since $I_2/Z(T)\cong A(n-5),$ we have
 $I_2^2\subseteq Z(T)\cong A(1)$ and hence  $cl(I_2) \leq 2.$
  Let $cl(I_2)=1.$ Therefore $[I_1,I_2]=I_1\cap I_2=Z(T)=Z(I_1)\cong A(1).$ Otherwise $[I_1,I_2]=0$  and since $  I_2$ is abelian,  $ I_2\subseteq Z(T)\cong A(1).$ It  is a contradiction, since $ \dim I_2\geq 2. $ Hence  $ I_2=Z(T) \oplus A$ in which $A\cong A(n-5)$ and $[I_1,I_2]=Z(T).$ Now $ Z(T)\subseteq I_1, A\cap I_1=0 $ and $ I_1\cap I_2 = Z(T) $ so $ T= I_1+I_2=I_1+Z(T)+A=I_1\rtimes A.$ Using the proof of Theorem \ref{48} $ (i),$  we have $ T\cong I_1\rtimes K\cong L_{6,11} $ in which $K\cong A(1)$ and $ [K,I_1]=Z(T) $ or $ T\cong I_1\rtimes K\cong L_{6,12} $ in which $K\cong A(1)$ and $ [K,I_1]=Z(T). $  \\ Now, let $ cl(I_2)=2.$ Since $I_2^2= I_1\cap I_2=Z(T)=Z(I_1)\cong A(1),$   \cite[Theorem 3.6]{ni3} implies $I_2\cong H(m)\oplus A(n-2m-5).$ First, assume that $ A(n-2m-5)=0.$ Then $n=2m+5$ and $ I_2\cong H(m).$ Using Theorem \ref{481} $ (ii),$ we can similarly obtain that $[I_1,I_2]=0$  and $ T=I_1\dotplus I_2, $ where $ I_2\cong H(m).  $ This is  cases $ (e), $ $(f),  $ $ (g) $ and  $ (h). $\\
  If  $  A(n-2m-5)\neq 0,$ then $n\neq 2m+5$ and hence $T=I_1+(K\oplus A)$ where $K\cong H(m)$ and  $A\cong A(n-2m-5)$ and $[I_1,K\oplus A]\subseteq Z(T)=Z(I_1).$ Using Theorem \ref{481} $ (ii),$ we have  $[I_1,K]=0.$ Now, we claim that $ [I_1,A]=Z(T)\cong A(1). $ Let $ [I_1,A]=0. $ Since $ [K,A]=0, $ we have  $ A\subseteq Z(T)=Z(I_1)=Z(K)=K^2\cong A(1).$ It is a contradiction, since $ A\cap K=0. $ Therefore $[I_1,A]=Z(T)  $ and hence $ T\cong (I_1\rtimes A)\dotplus K $ where $ A\cong A(n-2m-5)$ and  $ [I_1,A]=Z(T)=Z(I_1).$ Similar to cases $ (c) $ and  $ (d),$ one can obtain  $ A\cong A(1), $ so  $n=2m+6  $ and  $ [I_1,A]=Z(T). $ So $ T=(I_1\rtimes A)\dotplus K $ in which $ A\cong A(1) $ and $ [I_1,A]=Z(T). $ This is  cases $ (k) $ and  $ (l). $\\
      Now let $T/Z(T)\cong L_{5,5}\oplus  A(n-5).$
   Similarly, we can obtain all cases $ (m),$  $(n),$  $(p),$ $(q)$  and $(r),$ by using Proposition \ref{481}.
   The result follows.
\end{proof}
The capable stem Lie algebras of class $ 4 $ with the derived
subalgebra of dimension $ 3 $ are characterized  in the following.
\begin{lem}\label{cc1}
Let $ L $ be an $ n $-dimensional stem Lie algebra   such that $cl(L)=4$  and $\dim L^2=3.$ If $ L $ is isomorphic to one of the Lie algebras   $L\cong L_{5,6},  $  $ L\cong L_{5,7}, $ $L\cong L_{6,11}, $  $L\cong L_{6,12} $ or
$L\cong L_{6,13}, $ then $ L $ is   capable.
\end{lem}
\begin{proof}
By using a terminology of  \cite{cic,Gr}, we have $L_{6,15}=\langle x_1,\ldots,x_6| [x_1, x_2] = x_3, [x_1, x_3] = x_4, [x_1, x_4] = x_5=[x_2, x_3], [x_1, x_5] = x_6=[x_2, x_4]\rangle $ and
$L_{6,18}=\langle x_1,\ldots,x_6| [x_1, x_2] = x_3, [x_1, x_3] = x_4, [x_1, x_4] = x_5, [x_1, x_5] = x_6\rangle.$
It is clear to see that $Z(L_{6,15})=\langle x_6\rangle $ and $Z(L_{6,18})=\langle x_6\rangle,  $ so $L_{6,15} /\langle x_6\rangle\cong L_{5,6}$  and $L_{6,18}/\langle x_6\rangle\cong L_{5,7}.$ Hence they are capable. Consider $ H= \langle x_1,\ldots,x_7| [x_1, x_2] = x_3, [x_1, x_3] =x_4=[x_2,x_7],[x_1, x_4]=x_5=[x_3,x_7], [x_1,x_5]=x_6=[x_4,x_7]\rangle. $ Clearly, $ Z(H)= \langle x_6\rangle.$ Hence
$ L_{6,13}\cong H /\langle x_6\rangle.$ Therefore $L_{6,13}$ is capable. The capability of $L_{6,11}$ and $L_{6,12}$ are obtained by a similar way. Hence 
 the proof is completed.
\end{proof}
The following lemma  is a useful instrument in the next investigations.

\begin{lem}\label{4869}
The Lie algebra $L= \langle x_1,x_2,\ldots,x_{6},z|[x_1,x_i]=x_{i+1},[x_{2},x_{6}]=x_{4},[x_3,x_6]=[x_2,z]=[x_{6},z]=x_{5},1\leq i\leq 4\rangle $
is non-capable.
\end{lem}
\begin{proof}
Clearly, $ Z(L)=\langle x_5\rangle. $
We claim that $ x_{5}\in Z^{\wedge}(L). $  It is sufficient to see that $ x_i\wedge x_{5}= z\wedge x_{5}=0_{L\wedge L} $ for all $i, $  $1\leq i\leq 6.$ Since
\begin{align*}
&x_1\wedge x_{5}= x_1\wedge [x_{6},z]=[z,x_1]\wedge x_{6}-[x_{6},x_1]\wedge z=0_{L\wedge L},\\&
  x_3\wedge x_{5}=[x_1,x_{2}]\wedge x_{5}=x_1\wedge [x_2,x_{5}]-x_{2} \wedge [x_1,x_{5}]=0_{L\wedge L},\\&
   x_4\wedge x_{5}=[x_1,x_{3}]\wedge x_{5}=x_1\wedge [x_{3},x_{5}]-x_{3} \wedge [x_1,x_{5}]=0_{L\wedge L},\\& x_{6}\wedge x_{5}=x_{6}\wedge [x_1,x_{4}]=[x_4,x_{6}]\wedge x_1-[x_1,x_{6}]\wedge x_{4}=0_{L\wedge L},\\&
  z\wedge x_{5}=z\wedge [x_1,x_{4}]=[x_4,z]\wedge x_1-[x_1,z]\wedge x_{4}=0_{L\wedge L}
  \end{align*}
  and
  \begin{align*}
  & x_{2}\wedge x_{5}=x_{2}\wedge  [x_{6},z]=[z,x_{2}]\wedge x_{6}-[x_{6},x_2]\wedge z=-x_{5}\wedge x_{6}-\\&[x_{6},x_2]\wedge z=-x_4\wedge z=-[x_1,x_3]\wedge z=- x_1\wedge [x_3,z]+x_3\wedge [x_1,z]=0_{L\wedge L},
   \end{align*}
 $ x_{5}\in Z^{\wedge}(L),$ as required.
\end{proof}
Recall that a Lie algebra $L$ is called  unicentral if $Z^{*}(L)=Z(L).$
\begin{thm}\label{gk}
Let $ T $ be an $ n$-dimensional stem Lie algebra such that $cl(T)=4$  and $\dim T^2=3.$ Then $ T $ is non-capable if and only if $ n\geq 7.$
Moreover, $T $ is unicentral.
\end{thm}
\begin{proof}
Let  $ T $ be non-capable. Then  Theorem  \ref{lkl}  and Lemma \ref{cc1} imply $ n\geq 7.$ Conversely, let $ n\geq 7.$ Then by Proposition \ref{Hi}, Theorem  \ref{lkl} and Lemma \ref{4869},  $ T $ is non-capable. The result follows.
\end{proof}

In the  following theorem,  all capable stem Lie algebras of  class $ 4 $ with the  derived subalgebra of dimension $ 3$ are given.
\begin{thm}\label{p11}
  Let $ T $ be an $ n$-dimensional stem Lie algebra such that $cl(T)=4$  and $\dim T^2=3.$  Then $ T$ is capable if and only if $ T $ is isomorphic to one of the Lie algebras $L_{5,6},  $  $ L_{5,7}, $ $ L_{6,11}, $  $ L_{6,12} $ or
$ L_{6,13}. $
  \end{thm}
\begin{proof}
  Let $ T $ be  capable. By Theorem  \ref{lkl}, Lemma \ref{cc1}  and Theorem  \ref{gk}, $ T $  is isomorphic to one of the Lie algebras $L_{5,6},  $  $ L_{5,7}, $ $ L_{6,11}, $  $ L_{6,12} $ or
$ L_{6,13}. $ The converse holds by Lemma \ref{cc1}.
\end{proof}

The next theorem gives a necessary and sufficient condition for detecting  the capability of stem Lie algebras of class $ 4 $ with the derived subalgebra of dimension $ 3. $
\begin{thm}\label{fg}
Let $ T $ be an $ n$-dimensional stem Lie algebra such that $cl(T)=4$  and $\dim T^2=3.$ Then $ T$ is capable if and only if
$4\leq \dim (T/Z(T)) \leq 5.$
\end{thm}
\begin{proof}
The result follows from Lemma \ref{ggg} and Theorem \ref{p11}.
\end{proof}
The following result shows that each finite dimensional nilpotent Lie algebra $L$ of class $4$ with the derived subalgebra of dimension $ 3 $ has a decomposition into a stem 
Lie algebra $ T $ of class $4$ with $\dim T^2 = 3$ and an abelian Lie algebra. Moreover, there is a nice relationship between   the capability of $L$ and $ T. $
\begin{thm}\label{mo}
 Let $L$ be a finite dimensional nilpotent Lie algebra of class $4$ and $\dim L^2 = 3.$ Then $L = T \oplus A$ such that
  $Z(T ) = L^ 2 \cap Z(L) = L^4 = T^4$ and $ Z^*(L) = Z ^* (T ),$ where $A$ is an abelian Lie algebra.
 \end{thm}
 \begin{proof}
 By using \cite[Proposition 3.1]{pair}, $L = T\oplus A$ such that $Z(T ) = L^ 2 \cap Z(L)$ and $Z^*(L) = Z^* (T ),$ where $A$ is an abelian Lie algebra. Since $T$ is stem, Proposition \ref{134} implies $Z(T ) = T^4,$ as required.
 \end{proof}
Now, we are in the position to determine all  capable Lie algebras $ L $  of class $4 $ whit $ \dim L^2=3. $
\begin{thm}\label{2611}
Let $ L $ be an $n$-dimensional Lie algebra    such that $cl(L)=4$  and $\dim L^2=3.$ Then $L $ is capable if and only if $ L $ is isomorphic to one of the Lie algebras
 $L_{5,6}\oplus A(n-5),  $  $ L_{5,7}\oplus A(n-5), $ $ L_{6,11}\oplus A(n-6), $  $ L_{6,12}\oplus A(n-6) $ or
$ L_{6,13}\oplus A(n-6). $
\end{thm}
\begin{proof}
By using Theorem \ref{mo}, $ L=T\oplus A $ such that  $ Z(T) =L^2\cap Z(L)=L^4=T^4\cong A(1)$  and  $Z^{*}(L)=Z^{*}(T), $ where $ A $ is an abelian Lie algebra. Now, the result follows from Theorem \ref{p11}.
\end{proof}
The following result is obtained from Theorems \ref{fg} and \ref{2611}.
\begin{cor}\label{fg1}
Let $ L$ be a finite dimensional Lie algebra  of class $4  $ and $\dim L^2=3.$ Then $ L$ is capable if and only if
$4\leq \dim (L/Z(L)) \leq 5.$
\end{cor}

\end{document}